\documentclass[12pt]{article}
\usepackage{latexsym,amsfonts,amssymb,mathrsfs,float}
\usepackage{dsfont}
\usepackage{amsthm}
\usepackage[numbers]{natbib}
\usepackage{amsmath}
\usepackage{indentfirst}
\usepackage{bbm}

\usepackage{color}
\allowdisplaybreaks

\newtheorem{thm}{Theorem}[section]

\newtheorem{lem}[thm]{Lemma}
\newtheorem{pro}[thm]{Proposition}

\newtheorem{rem}[thm]{Remark}

\numberwithin{equation}{section}

\date{}

\begin{document}

\title{\bf Laws of the iterated logarithm for the range of once-reinforced random walk on $\mathbb{Z}$}

\author{Ze-Chun Hu$^{1}$, Renming Song$^{2}$ and Li Wang$^{1,}$\footnote{ Corresponding author}\\ \\
  {\small $^1$ College of Mathematics, Sichuan University,
Chengdu, 610065, China}\\
 {\small zchu@scu.edu.cn; wangli0@stu.scu.edu.cn}\\ \\
  {\small $^2$ Department of Mathematics,
University of Illinois Urbana-Champaign, Urbana, IL 61801, USA}\\
 {\small rsong@illinois.edu}}
\maketitle

\begin{abstract}

 In this paper, we study once-reinforced random walk
 on $\mathbb{Z}$, and prove sharp Chung-type and 
 Hartman-Wintner laws of the iterated logarithm for the range. 
 Let $R_n$ denote the range at time $n$ and let $c>0$  be the reinforcement parameter. Then, almost surely,
\begin{equation*}
\liminf_{n\to\infty}\frac{R_n}{\sqrt{n/\log\log n}}=\frac{\pi}{\sqrt{2}},
\qquad
\limsup_{n\to\infty}\frac{cR_n}{\sqrt{2n\log\log n}}=1.
\end{equation*}
 The proofs are based on the independent increment decomposition of the inverse range and sharp generating function estimates.

\end{abstract}

\smallskip

\noindent {\bf Keywords:} once-reinforced random walk, range, law of the iterated logarithm (LIL), Chung-type LIL, Hartman-Wintner LIL.

\smallskip

\noindent {\bf 2020 Mathematics Subject Classification:} 
60K35

\section{Introduction and main results}\label{sec-1}

Once-reinforced random walk (ORRW) was introduced by Davis \cite{Davis1990}.
It is a self-interacting random walk in which
each edge has weight one before its first traversal and weight $c>0$ afterwards. At each step the walk chooses an
incident edge with probability proportional to its current weight. In this paper we study laws of 
the iterated logarithm  (LIL) 
for the range of ORRW on $\mathbb Z$.
 This is a question
about fine range fluctuations, whereas much of the existing work on ORRW concerns recurrence and transience
(see,~e.g., \cite{CKS2020, CT2025, DKL2002, EK2026, KSS2021, KS2018, Vervoort2002} and the references therein).

Another classical model is the linearly edge-reinforced random walk (LERRW), 
introduced by Coppersmith and Diaconis \cite{CD1986}, in which the weight of an edge increases linearly with its number of traversals.
LERRW on finite graphs was studied in \cite{CD1986, ST2015}.
LERRW has also been studied on trees, two dimensional graphs, and $\mathbb Z^d$ for $d\ge3$
(see,~e.g., \cite{Pemantle1988, MR2009, DST2015}).
LERRW contrasts with ORRW, where each edge is reinforced only once. In dimension one,
this structure leads to the inverse range decomposition used throughout
the paper.

Let $X=(X_n)_{n\ge 0}$ be an ORRW on $\mathbb{Z}$ with reinforcement parameter $c>0$, started from $X_0=0$. We write
\begin{equation*}
M_n:=\max_{0\le j\le n}X_j,\qquad
m_n:=\min_{0\le j\le n}X_j,\qquad
\mathcal F_n:=\sigma(X_0,\dots,X_n).
\end{equation*}
The first step is symmetric, that is, $\mathbb P(X_1=1)=\mathbb P(X_1=-1)=1/2$. For $n\ge 1$, we have
\begin{eqnarray*}
\mathbb{P}(X_{n+1}=X_n+1\mid \mathcal F_n)= 1-\mathbb{P}(X_{n+1}=X_n-1\mid \mathcal F_n)=
\left\{
\begin{array}{cl}
\frac{1}{1+c} & \text{if}\ X_n=M_n,\\
\,\frac{1}{2} & \text{if}\ m_n<X_n<M_n,\\
\frac{c}{1+c} & \text{if}\ X_n=m_n.
\end{array}
\right.
\end{eqnarray*}
The cases $c>1$, $c=1$, and $0<c<1$
correspond, respectively, to the self attractive (reinforced), simple symmetric, and self repulsive regimes.
After the first step, the reinforcement affects the walk only at the two boundary points of the current range, while
inside the range the walk is symmetric. The growth
of the range is therefore governed by the successive times at which the walk pushes one of the two extrema outward. We
set
\begin{equation}\label{1.1}
R_n:=M_n-m_n=M_n+|m_n|.
\end{equation}
Then $R_n$ is the length of the visited interval, and  $R_n+1$ is the number of visited sites. The process $R_n$
increases only when the walk discovers a new site. We shall also use the inverse range
\begin{equation}\label{1.2}
S_k:=\inf\{n\ge 0:R_n=k\}
\end{equation}
as a basic tool in this paper.
It records the successive enlargements of the visited interval. Throughout the paper,
$\log$ denotes the natural logarithm and $\mathbb Z_+:=\{0,1,2,\dots\}$.

Serlet \cite{Serlet2018} studied a more general perturbed random walk on $\mathbb Z$ with parameters
$\beta,\gamma>0$. This walk is symmetric in the interior of its current range, whereas at its current maximum and minimum
it jumps outward with probabilities $1/(1+\beta)$ and $1/(1+\gamma)$ respectively. The symmetric special case $\beta=\gamma=c$ is precisely the ORRW. In \cite[Proposition~8]{Serlet2018}, Serlet obtained a scaling limit for the number $V_n$ of distinct sites visited by time $n$. More precisely, since $V_n=R_n+1$, his result implies that, as $n\to\infty$, the process $\left(R_{\lfloor nt\rfloor}/\sqrt n\right)_{t\ge0}$ converges in law to the process
$\left(\sup_{0\le s\le t}W_s-\inf_{0\le s\le t}W_s\right)_{t\ge0}$,
where $W=(W_t)_{t\ge0}$ is the unique solution of the perturbed Brownian equation in
\cite[Equation~(1)]{Serlet2018}. In the symmetric case $\beta=\gamma=c$, this equation becomes
\begin{equation*}
W_t=B_t-(c-1)\sup_{0\le s\le t}W_s-(c-1)\inf_{0\le s\le t}W_s,
\qquad t\ge0,
\end{equation*}
where $B=(B_t)_{t\ge0}$ is a standard Brownian motion.

The limiting process $W$ is a doubly perturbed Brownian motion. 
Chaumont, Doney and Hu \cite{CDH2000} established a Hartman-Wintner LIL for the running supremum of $W$ and a Chung-type LIL  for the running maximum of  $|W|$.
However, these continuum laws and the weak
invariance principle in \cite{Serlet2018} do not directly imply the almost sure laws for the discrete range studied in this paper. 

Pfaffelhuber and Stiefel \cite{PS2021} studied the range of ORRW on $\mathbb Z$. Using Tauberian theory, they proved that $\mathbb E[R_n^\ell]\sim C_\ell(c)n^{\ell/2}$ for every integer $\ell\ge 1$, with
explicit constants $C_\ell(c)$. Collevecchio and Tarr\`es \cite{CT2025}
developed a change-of-measure formula for ORRW on general graphs and obtained large-deviation upper bounds for the probability that
its range
on $\mathbb Z^d$ with $d\ge 2$ has
cardinality at most $un^{d/(d+2)}$.
In ~\cite{HMSW2026}, we (with Ting Ma) obtained asymptotic formulas for all moments of the range of ORRW on
$\mathbb Z_+$.
The moment orders of \cite{HMSW2026}
coincide with those obtained by Pfaffelhuber and Stiefel \cite{PS2021}, whereas the
asymptotic constants differ.

Serlet \cite{Serlet2013} studied the ORRW reflected at $0$ on $\mathbb Z_+$ and proved
a Chung-type LIL in the self attractive regime $c\ge 1$ (see \cite[Theorem 5]{Serlet2013}). He also obtained a
corresponding Hartman-Wintner upper bound under the additional restriction $c>5/4$ (see
\cite[Proposition 6]{Serlet2013}). Collevecchio et al.~\cite[Theorem~1.1]{CIT2025} sharpened Serlet's upper bound and proved the matching lower bound,
thereby establishing a Hartman-Wintner-type LIL 
for the position process for all $c>0$. The half-line results provide a useful comparison, but they do not directly yield LILs 
for the range on $\mathbb Z$. Indeed, on $\mathbb Z$, $R_n=M_n-m_n$ is determined jointly by the maximum $M_n$ and the
minimum $m_n$, whereas on $\mathbb Z_+$ the lower endpoint is fixed at $0$ and only the upper endpoint of the visited
interval can expand.

The aim of this paper is to establish Chung-type and Hartman-Wintner-type LILs 
for the range of ORRW on $\mathbb Z$. We now state the two main results, beginning with the Chung-type LIL.

\begin{thm}\label{thm-1.1}
Let $R_n$ be the range of the ORRW
on $\mathbb Z$ with reinforcement parameter $c>0$, as defined
in \eqref{1.1}. Then, almost surely,
\begin{equation*}
\liminf_{n\to\infty}\frac{R_n}{\sqrt{n/\log\log n}}
=\frac{\pi}{\sqrt{2}}.
\end{equation*}
\end{thm}

\begin{rem}\label{rem-1.1}
Serlet \cite{Serlet2013} uses the parametrization
\begin{equation*}
\beta:=\frac{1+r}{1-r},\qquad \frac{1-r}{2}=\frac{1}{1+\beta}.
\end{equation*}
Thus the self attractive regime $r\in[0,1)$ in \cite[Theorem 5]{Serlet2013} corresponds to $\beta\in[1,\infty)$. The
boundary transition probability in \cite{Serlet2013} is therefore $1/(1+\beta)$, the same form as $1/(1+c)$ in the
present paper, after identifying $\beta$ with $c$. In Serlet's Chung-type LIL for the perturbed reflecting random walk on
$\mathbb Z_+$, the sharp constant is $\pi/\sqrt{8}$ and is independent of the reinforcement parameter. Theorem
\ref{thm-1.1} establishes the corresponding range result on $\mathbb Z$, with constant
$\pi/\sqrt{2}=2(\pi/\sqrt{8})$, which is also independent of $c$.
\end{rem}

The second result is the Hartman-Wintner LIL for the range.

\begin{thm}\label{thm-1.2}
Let $R_n$ be the range of the ORRW
on $\mathbb Z$ with reinforcement parameter $c>0$, as defined
in \eqref{1.1}. Then, almost surely,
\begin{equation*}
\limsup_{n\to\infty}\frac{cR_n}{\sqrt{2n\log\log n}}=1.
\end{equation*}
\end{thm}

When $c=1$, the ORRW is the simple symmetric random walk. Hence, by Theorems \ref{thm-1.1} and
\ref{thm-1.2}, almost surely,
\begin{equation*}
\liminf_{n\to\infty}\frac{R_n}{\sqrt{n/\log\log n}}=\frac{\pi}{\sqrt{2}},\qquad
\limsup_{n\to\infty}\frac{R_n}{\sqrt{2n\log\log n}}=1,
\end{equation*}
 which
recover the corresponding laws for the range of the simple symmetric random walk on $\mathbb Z$.

The principal eigenvalue of the interior simple random walk determines the Chung logarithmic
exponent; the reinforcement contributes a positive prefactor to the single-long-increment lower
bound. In contrast, the leading small-deviation transform is $-c\mu$, and its optimization produces
$c^2/2$ for the inverse lower envelope. 
Thus the Chung constant is independent of $c$, whereas
\[
\limsup_{n\to\infty}\frac{R_n}{\sqrt{2n\log\log n}}=\frac1{c}.
\]

\begin{rem}\label{rem-1.2}
Collevecchio et al.
\cite[Theorem 1.1]{CIT2025} proved a Hartman-Wintner-type LIL for the position process on
$\mathbb Z_+$. For the walk on $\mathbb Z_+$, the range length is $\overline X_n:=\max_{0\le j\le n}X_j$. Since $X_n\le
\overline X_n$, the lower bound for $X_n$ in \cite[Theorem 1.1]{CIT2025} immediately yields the corresponding lower
bound for $\overline X_n$; the matching upper bound follows from \cite[Proposition 3.1]{CIT2025}. Thus, in the notation
of the present paper, almost surely,
\begin{equation*}
\limsup_{n\to\infty}\frac{c\overline X_n}{\sqrt{2n\log\log n}}=1.
\end{equation*}
Theorem \ref{thm-1.2} shows that the same normalization and sharp constant hold for the range on $\mathbb Z$.
\end{rem}

Following the strategy used for the half-line model in \cite[Section~6]{Serlet2013}, we work with the inverse range
process $(S_k)_{k\ge0}$ rather than directly with the range process $(R_n)_{n\ge0}$. In both settings, the inverse process has independent but non-identically distributed increments, and estimates for
this process are transferred to the range through the same monotone inverse relation. The essential difference lies in how the range grows. On $\mathbb Z_+$, the range is the maximum and is governed by a single moving boundary,
whereas on $\mathbb Z$, $R_n=M_n-m_n$ can grow at either endpoint. Our Chung-type LIL holds for every $c>0$, while
Serlet's half-line result was restricted to $c\ge 1$.
The genuinely harder part is to establish the matching inverse range bounds. For the half-line Chung-type LIL, 
Serlet obtained the matching upper bound for the range by comparison with a reflected
simple random walk (cf. \cite[Proposition~18]{Serlet2013}). For the half-line Hartman-Wintner LIL, 
the lower and upper bounds for the range follow from results for the position and
its maximum (cf. \cite[Theorem~1.1 and Proposition~3.1]{CIT2025} respectively).
Neither argument directly yields the sharp bounds required for the range $R_n=M_n-m_n$ on $\mathbb Z$.
Starting from the independent increment decomposition of Pfaffelhuber and Stiefel \cite{PS2021}, we therefore prove the matching
Chung lower bound for the inverse range by a sharp spectral estimate and the occurrence of a single exceptionally long
crossing. The matching Hartman-Wintner upper bound is more delicate: it is a collective small-deviation problem for many
independent but non-identically distributed waiting times, and we solve it by exponential tilting on sparse independent
blocks. Together with the adapted generating function estimates and monotone inversion, these arguments yield the exact
laws for the range on $\mathbb Z$.

The rest of the paper is organized as follows. In Section \ref{sec-2}, we present the inverse range decomposition, the
generating function estimates, and the monotone inversion lemma. In Sections \ref{sec-3} and \ref{sec-4}, we  prove
 the Chung-type LIL and Hartman-Wintner LIL respectively.

\section{Preliminaries}\label{sec-2}

This section establishes the preliminary lemmas needed for the proofs of Theorems \ref{thm-1.1} and \ref{thm-1.2}. We keep $c>0$ fixed and recall the range $R=(R_n)_{n\ge 0}$ and its inverse $S=(S_k)_{k\ge 0}$ from \eqref{1.1} and
\eqref{1.2}. We shall use the following representation by independent increments.

\begin{lem}\label{lem-2.1} (Pfaffelhuber and Stiefel
\cite{PS2021})
For each $i\ge 1$, set
\begin{equation*}
T_i:=S_{i+1}-S_i.
\end{equation*}
Then the random variables $(T_i)_{i\ge 1}$ are independent and may be represented as
\begin{equation}\label{2.1}
T_i=1+\sum_{j=1}^{Y_i}\bigl(1+\tau_i^{(j)}\bigr),\quad i\ge 1,
\end{equation}
where the random variables $(Y_i,\tau_i^{(j)}:i\ge 1,\ j\ge 1)$ are mutually independent. 
Here, $Y_i$ has the shifted
geometric distribution with parameter $1/(1+c)$, i.e., 
\begin{equation*}
\mathbb{P}(Y_i=m)=\frac{1}{1+c}\Bigl(\frac{c}{1+c}\Bigr)^m,\quad m=0,1,2,\dots.
\end{equation*}
For each $i\ge 1$, let $\tau_i$ be the hitting time of $\{-1,i-1\}$ of a simple symmetric random walk started from $0$.
The variables $(\tau_i^{(j)})_{j\ge 1}$ are independent copies of $\tau_i$. Consequently,
\begin{equation}\label{2.2}
S_k=1+\sum_{i=1}^{k-1}T_i,\quad k\ge 1.
\end{equation}
\end{lem}

Next, we derive hyperbolic and trigonometric forms of the generating function of $S_k$.

\begin{lem}\label{lem-2.2}
For any $k\ge 1$ and $\theta>0$, define
\begin{equation*}
\Phi_k(\theta):=\mathbb{E}\bigl[(\operatorname{sech}\theta)^{S_k}\bigr].
\end{equation*}
Then
\begin{equation}\label{2.3}
\Phi_k(\theta)
=(\operatorname{sech}\theta)^k\prod_{i=1}^{k-1}
\frac{1}{1+c\,\tanh\theta\,\tanh(i\theta/2)}.
\end{equation}

Furthermore, if $0<\theta<\min\{\pi/k,\pi/2\}$ and $c\tan\theta\,\tan((k-1)\theta/2)<1$, then 
\begin{equation*}
\Psi_k(\theta):=\mathbb{E}\bigl[(\sec\theta)^{S_k}\bigr]
\end{equation*}
is finite and satisfies
\begin{equation}\label{2.4}
\Psi_k(\theta)
=(\sec\theta)^k\prod_{i=1}^{k-1}
\frac{1}{1-c\,\tan\theta\,\tan(i\theta/2)}.
\end{equation}
\end{lem}

\begin{rem}\label{rem-2.3}
For comparison, let
\begin{equation*}
T_n^+:=\inf\{k\ge0:X_k=n\}
\end{equation*}
be the hitting time of level $n$ for the ORRW $(X_k)_{k\ge0}$ on $\mathbb{Z}_+$ reflected at $0$, with reinforcement
parameter $c$. By \cite[Proposition~16, equations~(45)--(46)]{Serlet2013}, for every $\lambda>0$,
\begin{equation}\label{2.5}
\mathbb{E}\left[\left(\frac{1}{\cosh\lambda}\right)^{T_n^+}\right]
=\frac{1}{\cosh^n\lambda}\prod_{j=1}^{n-1}
\frac{1}{1+c\,\tanh\lambda\,\tanh(j\lambda)},
\end{equation}
and, for $\lambda\in(-\pi/(2n),\pi/(2n))$ such that $c\tan\lambda\tan(j\lambda)<1$ for all $1\le j\le n-1$,
\begin{equation}\label{2.6}
\mathbb{E}\left[\left(\frac{1}{\cos\lambda}\right)^{T_n^+}\right]
=\frac{1}{\cos^n\lambda}\prod_{j=1}^{n-1}
\frac{1}{1-c\,\tan\lambda\,\tan(j\lambda)}.
\end{equation}
Comparison of \eqref{2.5} with \eqref{2.3} and of \eqref{2.6} with \eqref{2.4} shows that
the terms $\tanh(j\lambda)$ and $\tan(j\lambda)$ for the half-line model are replaced   by
$\tanh(j\theta/2)$ and $\tan(j\theta/2)$ respectively, in the present setting.
\end{rem}

\begin{proof}[ Proof of Lemma \ref{lem-2.2}]
The assertion is immediate for $k=1$, since $S_1=1$. We henceforth assume that $k\ge 2$. After translating the simple
symmetric random walk in the definition of $\tau_i$ by one unit, $\tau_i$ becomes the hitting time of $\{0,i\}$
of
a simple symmetric random walk started from $1$. Hence, by \cite[Lemma~12, equation~(28)]{Serlet2013},
\begin{equation}\label{2.7}
\mathbb{E}[(\operatorname{sech}\theta)^{1+\tau_i}]=1-\tanh\theta\,\tanh(i\theta/2),\quad 1\le i\le k-1.
\end{equation}
The generating function of $Y_i$ is
\begin{equation}\label{2.8}
\mathbb{E}[z^{Y_i}]=\frac{1}{1+c}\sum_{m=0}^{\infty}\left(\frac{c}{1+c}z\right)^m=\frac{1}{1+c-cz},
\quad 0\le z<\frac{1+c}{c}.
\end{equation}
Since $\mathbb{E}[(\operatorname{sech}\theta)^{1+\tau_i}]<1$, we may apply \eqref{2.8} with
$z=\mathbb{E}[(\operatorname{sech}\theta)^{1+\tau_i}]$. Thus, by Lemma \ref{lem-2.1}, conditioning on $Y_i$, and \eqref{2.7}, we have
\begin{align}
&\mathbb{E}[(\operatorname{sech}\theta)^{T_i}]
=\operatorname{sech}\theta\cdot \mathbb{E}\left[\mathbb{E}\left[\prod_{j=1}^{Y_i}(\operatorname{sech}\theta)^{1+\tau_i^{(j)}}
\,\middle|\,Y_i\right]\right] \nonumber\\
&=\operatorname{sech}\theta\cdot\mathbb{E}\left[\left(\mathbb{E}[(\operatorname{sech}\theta)^{1+\tau_i}]\right)^{Y_i}\right] 
=\frac{\operatorname{sech}\theta}{1+c-c\,\mathbb{E}[(\operatorname{sech}\theta)^{1+\tau_i}]} \nonumber\\
&=\frac{\operatorname{sech}\theta}{1+c\,\tanh\theta\,\tanh(i\theta/2)}\label{2.8-a}.
\end{align}
Using \eqref{2.2} and the independence of the $T_i$'s, we obtain
\begin{align*}
\Phi_k(\theta)
&=\operatorname{sech}\theta\cdot \prod_{i=1}^{k-1}\mathbb{E}[(\operatorname{sech}\theta)^{T_i}]
=(\operatorname{sech}\theta)^k\prod_{i=1}^{k-1}\frac{1}{1+c\,\tanh\theta\,\tanh(i\theta/2)},
\end{align*}
which is \eqref{2.3}.

We now prove the trigonometric form. 
Fix $1\le i\le k-1$, and let $(Z_n)_{n\ge0}$ be the simple symmetric
random walk in the definition of $\tau_i$, with natural filtration $(\mathcal{F}_n)_{n\ge0}$. Define
\begin{equation*}
M_n:=(\sec\theta)^n\cos\left(\theta\left(Z_n-\frac{i-2}{2}\right)\right),
\quad n\ge0.
\end{equation*}
Then $(M_n)_{n\ge0}$ is a martingale, since
\begin{align*}
&\mathbb{E}[M_{n+1}\mid\mathcal{F}_n]
=\frac{(\sec\theta)^{n+1}}{2}\left[
\cos\left(\theta\left(Z_n+1-\frac{i-2}{2}\right)\right)
+\cos\left(\theta\left(Z_n-1-\frac{i-2}{2}\right)\right)\right]\\
&=(\sec\theta)^{n+1}\cos\theta\,
\cos\left(\theta\left(Z_n-\frac{i-2}{2}\right)\right)
=M_n.
\end{align*}
Applying the optional stopping theorem to $(M_n)_{n\ge0}$ at $\tau_i\wedge n$ and using $Z_0=0$, we get
\begin{equation}\label{2.9}
\mathbb{E}[M_{\tau_i\wedge n}]=M_0=\cos(i\theta/2-\theta).
\end{equation}
Because
 $1\leq i\le k-1$ and $0<\theta<\pi/k$,
we have $0<i\theta/2<\pi/2$. Since $\tau_i$ is the first hitting time of $\{-1,i-1\}$, we have
$Z_{\tau_i\wedge n}\in\{-1,0,\ldots,i-1\}$. Hence
\begin{equation*}
\left|Z_{\tau_i\wedge n}-\frac{i-2}{2}\right|\le\frac{i}{2}.
\end{equation*}
Since the cosine function is even and decreasing on $[0,\pi/2]$, we have
\begin{equation*}
\cos\left(\theta\left(Z_{\tau_i\wedge n}-\frac{i-2}{2}\right)\right)
\ge\cos(i\theta/2)>0.
\end{equation*}
It follows that
\begin{equation}\label{2.10}
M_{\tau_i\wedge n}=(\sec\theta)^{\tau_i\wedge n}
\cos\left(\theta\left(Z_{\tau_i\wedge n}-\frac{i-2}{2}\right)\right)
\ge \cos(i\theta/2)(\sec\theta)^{\tau_i\wedge n}.
\end{equation}
Taking expectations in \eqref{2.10} and using \eqref{2.9}, we obtain
\begin{equation}\label{2.11}
\cos(i\theta/2-\theta)\ge \cos(i\theta/2)\,\mathbb{E}[(\sec\theta)^{\tau_i\wedge n}].
\end{equation}
Since $\tau_i<\infty$ almost surely and $\sec\theta>1$, we have
$(\sec\theta)^{\tau_i\wedge n}\uparrow(\sec\theta)^{\tau_i}$. Hence, by the monotone convergence theorem and \eqref{2.11}, we get
\begin{equation*}
\mathbb{E}[(\sec\theta)^{\tau_i}]
=\lim_{n\to\infty}\mathbb{E}[(\sec\theta)^{\tau_i\wedge n}]
\le\frac{\cos(i\theta/2-\theta)}{\cos(i\theta/2)}<\infty.
\end{equation*}

The preceding estimate shows that $(\sec\theta)^{\tau_i}$ is integrable. Moreover,
$M_{\tau_i\wedge n}\to M_{\tau_i}$ almost surely, while
\begin{equation*}
|M_{\tau_i\wedge n}|\le (\sec\theta)^{\tau_i\wedge n}\le (\sec\theta)^{\tau_i}.
\end{equation*}
We may therefore let $n\to\infty$ in \eqref{2.9} by the dominated convergence theorem. Since $Z_{\tau_i}\in\{-1,i-1\}$, the argument of the cosine in $M_{\tau_i}$ is either 
 $-i\theta/2$ or $i\theta/2$, and so  the cosine factor equals $\cos(i\theta/2)$. 
Consequently,
\begin{equation*}
\cos(i\theta/2-\theta)=\lim_{n\to\infty}\mathbb{E}[M_{\tau_i\wedge n}]
=\mathbb{E}[M_{\tau_i}]=\cos(i\theta/2)\,\mathbb{E}[(\sec\theta)^{\tau_i}].
\end{equation*}
 It follows that
\begin{equation}\label{2.12}
\mathbb{E}[(\sec\theta)^{1+\tau_i}]=\sec\theta\cdot\mathbb{E}[(\sec\theta)^{\tau_i}]
=\sec\theta\,\frac{\cos(i\theta/2-\theta)}{\cos(i\theta/2)}
=1+\tan\theta\,\tan(i\theta/2).
\end{equation}
Since $\tan$ is increasing on $(0,\pi/2)$, the assumption
$c\tan\theta\,\tan((k-1)\theta/2)<1$ implies
\begin{equation*}
c\tan\theta\,\tan(i\theta/2)<1,\quad \forall 1\le i\le k-1.
\end{equation*}
Consequently,
\begin{equation*}
\frac{c}{1+c}\,\mathbb{E}[(\sec\theta)^{1+\tau_i}]
=\frac{c}{1+c}\bigl(1+\tan\theta\,\tan(i\theta/2)\bigr)<1.
\end{equation*}
Hence, we may apply \eqref{2.8} with $z=\mathbb{E}[(\sec\theta)^{1+\tau_i}]$. By Lemma
\ref{lem-2.1}, conditioning on $Y_i$, and \eqref{2.12}, we get
\begin{align*}
&\mathbb{E}[(\sec\theta)^{T_i}]
=\sec\theta\cdot\mathbb{E}\left[\mathbb{E}\left[\prod_{j=1}^{Y_i}(\sec\theta)^{1+\tau_i^{(j)}}\,\middle|\,Y_i\right]\right]\\
&=\sec\theta\cdot\mathbb{E}\left[\left(\mathbb{E}[(\sec\theta)^{1+\tau_i}]\right)^{Y_i}\right]
=\frac{\sec\theta}{1+c-c\,\mathbb{E}[(\sec\theta)^{1+\tau_i}]}\\
&=\frac{\sec\theta}{1-c\,\tan\theta\,\tan(i\theta/2)}.
\end{align*}
Using \eqref{2.2} and the independence of the $T_i$'s, we have
\begin{equation*}
\Psi_k(\theta)=\sec\theta\cdot\prod_{i=1}^{k-1}\mathbb{E}[(\sec\theta)^{T_i}]
=(\sec\theta)^k\prod_{i=1}^{k-1}\frac{1}{1-c\,\tan\theta\,\tan(i\theta/2)},
\end{equation*}
which is \eqref{2.4}. The proof is complete.
\end{proof}

The following elementary inversion converts estimates for the inverse range at level $k$ into estimates for the range at
time $n$.

\begin{lem}\label{lem-2.4}
For the range $R=(R_n)_{n\ge 0}$ and its inverse $S=(S_k)_{k\ge 0}$ defined in \eqref{1.1} and \eqref{1.2}, the
following inversion relations hold.

(i) If there exists a constant $a\in(0,\infty)$ such that
\begin{equation}\label{i.1}
\limsup_{k\to\infty}\frac{S_k}{k^2\log\log k}=a\quad{\rm a.s.}, 
\end{equation}
then 
\begin{equation}\label{i.2}
\liminf_{n\to\infty}\sqrt{\frac{\log\log n}{n}}\,R_n=a^{-1/2}\quad{\rm a.s.}
\end{equation}

(ii) If there exists a constant $b\in(0,\infty)$ such that
\begin{equation}\label{ii.1}
\liminf_{k\to\infty}\frac{S_k\log\log k}{k^2}=b\quad{\rm a.s.}, 
\end{equation}
then 
\begin{equation}\label{ii.2}
\limsup_{n\to\infty}\frac{R_n}{\sqrt{n\log\log n}}=b^{-1/2}\quad{\rm a.s.}
\end{equation}
\end{lem}

\begin{proof}
Let
\begin{equation*}
L(x):=\log\log x
\end{equation*}
for large $x$. By the definition of $S$ and the monotonicity of $R$, for all sufficiently large $n$,
\begin{equation}\label{2.13}
S_{R_n}\le n<S_{R_n+1}, \quad \text{ and } R_n=k \ \text{whenever} \ S_{k}\le n<S_{k+1}.
\end{equation}
Since the range can increase by at most one at each step, $R_n\le n$. By Lemma \ref{lem-2.1}, $S_k<\infty$ for all
$k$ almost surely. Hence
 $R_n\to\infty$ a.s.

$(i)$ Define the subsequence
\begin{equation*}
n_k:=S_k-1.
\end{equation*}
On the interval $S_{k-1}\le n<S_k$, by \eqref{2.13},
we know that the range is $k-1$.
Since $\sqrt{L(n)/n}$ is eventually decreasing as $n\to\infty$, the minimum of $R_n\,\sqrt{L(n)/n}$ on this interval is attained at 
$n_k$. Thus,
\begin{equation}\label{2.15}
\liminf_{n\to\infty}R_n\,\sqrt{\frac{L(n)}{n}}=\liminf_{k\to\infty}(k-1)\sqrt{\frac{L(n_k)}{n_k}}\quad {\rm a.s.}
\end{equation}
Fix $\varepsilon\in(0,a)$. By \eqref{i.1}, almost surely for all sufficiently large $k$,
\begin{equation}\label{2.20}
S_k\le(a+\varepsilon)k^2L(k).
\end{equation}
Since $R_n\le n$, we have $S_k\ge k$. Together with \eqref{2.20}, this implies $L(S_k-1)/L(k)\to1$, $(S_k-1)/S_k\to1$.
Therefore, by \eqref{2.15}, we have
\begin{align*}
 &\liminf_{n\to\infty}R_n\sqrt{\frac{L(n)}{n}}=\liminf_{k\to\infty}(k-1)\sqrt{\frac{L(n_k)}{n_k}}\\
 &=\liminf_{k\to\infty}\frac{k-1}{k}\sqrt{\frac{L(S_k-1)}{L(k)}}\sqrt{\frac{S_k}{S_k-1}}\sqrt{\frac{k^2L(k)}{S_k}}\\
 &=\liminf_{k\to\infty}(1+o(1))\left(\frac{S_k}{k^2L(k)}\right)^{-1/2}
 =\left(\limsup_{k\to\infty}\frac{S_k}{k^2L(k)}\right)^{-1/2}
 =a^{-1/2}\quad {\rm a.s.}
\end{align*}
This proves \eqref{i.2}.

 $(ii)$ Define the subsequence
\begin{equation*}
N_k:=S_k.
\end{equation*}
On the interval $S_{k}\le n<S_{k+1}$, by \eqref{2.13},
we know that the range is $k$.
Since $1/\sqrt{nL(n)}$ is eventually decreasing as $n\to\infty$, the maximum of $R_n/\sqrt{nL(n)}$ on this interval is attained 
at $N_k$. Thus,
\begin{equation}\label{2.21}
\limsup_{n\to\infty}\frac{R_n}{\sqrt{nL(n)}}=\limsup_{k\to\infty}\frac{k}{\sqrt{N_kL(N_k)}}\quad {\rm a.s.}
\end{equation}
Since $R_{S_k}=k$ and $R_n\le n$,
we have $N_k=S_k\ge k$, and thus $L(N_k)\ge L(k)$ for all $k$.
By \eqref{2.21} and \eqref{ii.1}, we obtain
\begin{align}\label{2.22}
&\limsup_{n\to\infty}\frac{R_n}{\sqrt{nL(n)}}=\limsup_{k\to\infty}\frac{k}{\sqrt{N_kL(N_k)}}\nonumber\\
&\le\limsup_{k\to\infty}\left(\frac{S_kL(k)}{k^2}\right)^{-1/2}
=\left(\liminf_{k\to\infty}\frac{S_kL(k)}{k^2}\right)^{-1/2}
=b^{-1/2}\quad {\rm a.s.}
\end{align}
For the reverse inequality, by \eqref{ii.1}, 
we can choose a subsequence $k_j\to\infty$ such that
\begin{equation}\label{2.23}
\frac{S_{k_j}L(k_j)}{k_j^2}\to b \quad {\rm a.s.}
\end{equation}
Along this subsequence, for any $\varepsilon>0$, $k_j\le N_{k_j}=S_{k_j}\le (b+\varepsilon)k_j^2/L(k_j)$ for sufficiently large $j$, and thus $L(N_{k_j})/L(k_j)\to1$.
By \eqref{2.21} and \eqref{2.23}, we have
\begin{align}\label{2.24}
&\limsup_{n\to\infty}\frac{R_n}{\sqrt{nL(n)}}=\limsup_{k\to\infty}\frac{k}{\sqrt{N_kL(N_k)}}\nonumber\\
&\ge\lim_{j\to\infty}\frac{k_j}{\sqrt{N_{k_j}L(N_{k_j})}}
=\lim_{j\to\infty}\left(\frac{S_{k_j}L(k_j)}{k_j^2}\right)^{-1/2}\sqrt{\frac{L(k_j)}{L(N_{k_j})}}
= b^{-1/2}\quad {\rm a.s.}
\end{align}
Combining \eqref{2.22} and \eqref{2.24}, we obtain \eqref{ii.2}.
The proof is complete.
\end{proof}

\section{Proof of Theorem \ref{thm-1.1}}\label{sec-3}

 In this section, we give the proof of Theorem \ref{thm-1.1}. We need the following two propositions, whose proofs are postponed to Sections \ref{sec-3.2} and \ref{sec-3.3} respectively.
Recall that $S_k$ is 
the inverse range of the once-reinforced random walk on $\mathbb Z$ with reinforcement parameter $c>0$.

\begin{pro}\label{pro-3.1}
It holds that
\begin{equation}\label{pro-3.1-a}
\limsup_{k\to\infty}\frac{S_k}{k^2\log\log k}\le \frac{2}{\pi^2}\quad {\rm a.s.}
\end{equation}
\end{pro}

\begin{pro}\label{pro-3.2}
It holds that 
\begin{equation}\label{pro-3.2-a}
\limsup_{k\to\infty}\frac{S_k}{k^2\log\log k}\ge \frac{2}{\pi^2}\quad {\rm a.s.}
\end{equation}
\end{pro}

\smallskip

\begin{proof}[Proof of Theorem \ref{thm-1.1}]
By Propositions \ref{pro-3.1} and \ref{pro-3.2}, we get
\begin{equation*}
\limsup_{k\to\infty}\frac{S_k}{k^2\log\log k}=\frac{2}{\pi^2}\quad {\rm a.s.},
\end{equation*}
 which together with Lemma \ref{lem-2.4}(i) implies that
\begin{equation*}
\liminf_{n\to\infty}\sqrt{\frac{\log\log n}{n}}\,R_n
= \left(\frac{2}{\pi^2}\right)^{-1/2}=\frac{\pi}{\sqrt 2}\quad {\rm a.s.}
\end{equation*}
The proof is complete.
\end{proof}

\subsection{Proof of Proposition \ref{pro-3.1}}\label{sec-3.2}

To prove Proposition \ref{pro-3.1}, we first establish a uniform bound for the trigonometric generating function  defined by \eqref{2.4}.
 The following lemma adapts the corresponding estimate in deriving the bound in \cite[Theorem~17, equation~(54)]{Serlet2013} to our setting.

\begin{lem}\label{lem-3.3}
For $\mu\in(0,\pi)$, 
there exists an integer $k_0=k_0(\mu,c)$ such that $\Psi_k(\mu/k)$ is well-defined for every
$k\ge k_0$ and
\begin{equation*}
\sup_{k\ge k_0}\Psi_k(\mu/k)<\infty.
\end{equation*}
\end{lem}

\begin{proof}
Since $\mu\in  (0,\pi)$, we have $\mu/k<\pi/k$ for every $k$, and $\mu/k<\pi/2$ for all sufficiently large $k$. Also, since
$\tan(\mu/k)\sim \mu/k$ and $\lim_{k\to\infty}\tan((k-1)\mu/2k)=\tan(\mu/2)<\infty$, there exists $k_0=k_0(\mu,c)$ such that
\begin{equation}\label{3.1}
c\tan(\mu/k)\tan\bigl((k-1)\mu/2k\bigr)\le \frac{1}{2},
\quad  \forall k\ge k_0.
\end{equation}
Thus we may choose $k_0=k_0(\mu,c)$ so that the assumptions of Lemma \ref{lem-2.2} are satisfied with
$\theta=\mu/k$ for all $k\ge k_0$. Then, by \eqref{2.4}, $\Psi_k(\mu/k)$ is well-defined for any $k\ge k_0$ and
\begin{equation*}
\Psi_k(\mu/k)
=\bigl(\sec(\mu/k)\bigr)^k\prod_{i=1}^{k-1}
\frac{1}{1-c\tan(\mu/k)\tan(i\mu/2k)}.
\end{equation*}
Taking logarithms, we obtain
\begin{equation*}
\log\Psi_k(\mu/k)
=k\log\sec(\mu/k)-\sum_{i=1}^{k-1}\log\Bigl(1-c\tan(\mu/k)\tan(i\mu/2k)\Bigr).
\end{equation*}
Since $\log\sec x=x^2/2+O(x^4)$ as $x\to0$, we have $k\log\sec(\mu/k)=O(k^{-1})$. Moreover, since $\tan$ is increasing on
$[0,\mu/2]$, it follows from \eqref{3.1} that
\begin{equation*}
0\le c\tan(\mu/k)\tan(i\mu/2k)\le \frac{1}{2},
\quad  \forall 1\le i\le k-1,
\end{equation*}
and therefore, by $-\log(1-x)\le 2x$ on $[0,1/2]$,
\begin{equation*}
\log\Psi_k(\mu/k)
\le O(1)+2c\tan(\mu/k)\sum_{i=1}^{k-1}\tan(i\mu/2k).
\end{equation*}
Since $\tan(\mu/k)\sim \mu/k$, it remains to control the Riemann sum. We have
\begin{equation*}
\frac{\mu}{k}\sum_{i=1}^{k-1}\tan(i\mu/2k)\longrightarrow
\mu\int_0^1\tan(\mu x/2)\,dx=-2\log\cos(\mu/2)<\infty.
\end{equation*}
It follows that $\log\Psi_k(\mu/k)$ is bounded above uniformly in $k\ge k_0$. This proves
$\sup_{k\ge k_0}\Psi_k(\mu/k)<\infty$. The proof is complete.
\end{proof}

\begin{proof}[Proof of Proposition \ref{pro-3.1}]
Fix $q>1$, let $k_n:=\lfloor q^n\rfloor$, and fix $\mu\in(0,\pi)$. By Lemma \ref{lem-3.3}, there exist $k_0=k_0(\mu,c)$
and $C=C(\mu,c)<\infty$ such that
\begin{equation*}
\Psi_k(\mu/k)\le C, \quad  \forall k\ge k_0.
\end{equation*}
Recall that $\Psi_{k_n}(\mu/k_n)=\mathbb{E}[(\sec(\mu/k_n))^{S_{k_n}}]$. Hence, for any $A>0$ and all large $n$, since
$\sec(\mu/k_n)>1$, by Markov's inequality,
\begin{align*}
\mathbb{P}\bigl(S_{k_n}\ge A k_n^2\log n\bigr)
&=\mathbb{P}\Bigl((\sec(\mu/k_n))^{S_{k_n}}
\ge (\sec(\mu/k_n))^{A k_n^2\log n}\Bigr)\\
&\le (\cos(\mu/k_n))^{A k_n^2\log n}\Psi_{k_n}(\mu/k_n)\\
&\le C\exp\Bigl(Ak_n^2\log n\cdot\log\cos(\mu/k_n)\Bigr)
=Cn^{-A\mu^2/2+o(1)}.
\end{align*}
The last equality follows from the expansion $\log\cos x=-x^2/2+O(x^4)$ as $x\to 0$. Since $k_n\sim q^n$, we
have $\log n/k_n^2\to 0$, and hence
\begin{equation*}
Ak_n^2\log n\cdot\log\cos(\mu/k_n)
=-\frac{A\mu^2}{2}\log n+o(1)
=\left(-\frac{A\mu^2}{2}+o(1)\right)\log n.
\end{equation*}
Thus the series $\sum_{n=1}^{\infty}Cn^{-A\mu^2/2+o(1)}$ is summable as soon as $A>2/\mu^2$. It follows from the Borel-Cantelli lemma that
\begin{equation*}
\limsup_{n\to\infty}\frac{S_{k_n}}{k_n^2\log n}\le A
\quad {\rm a.s.}
\end{equation*}
Since $A>2/\mu^2$ is arbitrary and $\log\log k_n\sim \log n$,
taking $A>2/\mu^2$ to be rational
and letting
$\mu\uparrow\pi$ along a sequence yield
\begin{equation*}
\limsup_{n\to\infty}\frac{S_{k_n}}{k_n^2\log\log k_n}\le \frac{2}{\pi^2}
\quad\text{a.s.}
\end{equation*}

Now fix $k$ large and choose $n$ such that $k_n\le k<k_{n+1}$. Since $S_k$ is increasing and $k^2\log\log k\ge
k_n^2\log\log k_n$ for all large $k$,
\begin{equation*}
\frac{S_k}{k^2\log\log k}
\le \frac{S_{k_{n+1}}}{k_n^2\log\log k_n}
=\Bigl(\frac{k_{n+1}}{k_n}\Bigr)^2\frac{\log\log k_{n+1}}{\log\log k_n}
\cdot \frac{S_{k_{n+1}}}{k_{n+1}^2\log\log k_{n+1}}.
\end{equation*}
Taking the limsup and using $k_{n+1}/k_n\to q$ and $\log\log k_{n+1}/\log\log k_n\to 1$, we obtain
\begin{equation*}
\limsup_{k\to\infty}\frac{S_k}{k^2\log\log k}\le q^2\frac{2}{\pi^2}
\quad {\rm a.s.}
\end{equation*}
Finally, letting $q\downarrow1$ along a sequence yields the result.
The proof is complete.
\end{proof}

\subsection{Proof of Proposition \ref{pro-3.2}}\label{sec-3.3}

We prove the lower Chung bound through a tail estimate for the exit time of a killed simple random walk.
 Recall that for each $i\geq 1$,  $\tau_i$ stands for the hitting time of $\{-1,i-1\}$ of a simple symmetric random walk started from $0$ (cf. Lemma \ref{lem-2.1}).

\begin{lem}\label{lem-3.4}
For any integers $i\ge 2$ and $m\ge 1$, we have
\begin{align}\label{lem-3.4-a}
\mathbb{P}(\tau_i>2m)\ge \frac{2}{i}\cos^{2m}\!\Bigl(\frac{\pi}{i}\Bigr).
\end{align}
\end{lem}

\begin{proof}
Fix integers $i\ge 2$ and $m\ge 1$.
Let $Z$ be a simple symmetric random walk, and write
$\mathbb{P}_x$ for its law when $Z_0=x$. Set
\begin{equation*}
\sigma_i:=\inf\{n\ge 0:Z_n\in\{0,i\}\}.
\end{equation*}
Under $\mathbb{P}_1$, with $\widetilde Z_n:=Z_n-1$, the process $\widetilde Z$ is a simple symmetric random walk started
from $0$, and $\sigma_i=\inf\{n\ge 0:\widetilde Z_n\in\{-1,i-1\}\}$. Consequently, the law of $\tau_i$ is the
same as that of $\sigma_i$ under $\mathbb{P}_1$.

Let $E:=\{1,\dots,i-1\}$ and
$K$
be the killed transition kernel on $E$,
\begin{align*}
K(x,y):=\mathbb{P}_x(Z_1=y,\sigma_i>1)
=\frac{1}{2}\mathbf 1_{\{|x-y|=1\}},\quad x,y\in E.
\end{align*}
Then
\begin{equation*}
\mathbb{P}_x(Z_n=y,\sigma_i>n)=K^n(x,y).
\end{equation*}
For $1\le r\le i-1$, define
\begin{align*}
\phi_r(x):=\sqrt{\frac{2}{i}}\sin\!\Bigl(\frac{r\pi x}{i}\Bigr), \quad x\in E.
\end{align*}
By the one-dimensional case of \cite[Section~8.2]{LL2010},  the functions $\{\phi_r, r=1,\ldots,i-1\}$ form an orthonormal basis of $\ell^2(E)$, with eigenvalues $\cos(r\pi/i)$ for $K$.
By the finite dimensional spectral theorem
(see,~e.g., \cite[Chapter~XVI.1]{Feller1968}),
\begin{equation*}
K^n(x,y)=\sum_{r=1}^{i-1}\cos^n\!\Bigl(\frac{r\pi}{i}\Bigr)\phi_r(x)\phi_r(y),
\end{equation*}
that is, for $x,y\in E$,
\begin{align*}
\mathbb{P}_x(Z_n=y,\sigma_i>n)
&=\frac{2}{i}\sum_{r=1}^{i-1}
\cos^n\!\Bigl(\frac{r\pi}{i}\Bigr)
\sin\!\Bigl(\frac{r\pi x}{i}\Bigr)
\sin\!\Bigl(\frac{r\pi y}{i}\Bigr).
\end{align*}
Summing over all interior states at even time $2m$, we obtain
\begin{equation}\label{3.2}
\begin{aligned}
&\mathbb{P}(\tau_i>2m)
=\mathbb{P}_1(\sigma_i>2m)
=\sum_{y\in E}K^{2m}(1,y)\\
&=\frac{2}{i}\sum_{r=1}^{i-1}
\cos^{2m}\!\Bigl(\frac{r\pi}{i}\Bigr)
\sin\!\Bigl(\frac{r\pi}{i}\Bigr)
\sum_{y=1}^{i-1}\sin\!\Bigl(\frac{r\pi y}{i}\Bigr).
\end{aligned}
\end{equation}
The elementary identity $\sum_{y=1}^{i-1}\sin(y\theta)=\sin((i-1)\theta/2)\sin(i\theta/2)/\sin(\theta/2)$, applied at
$\theta=r\pi/i$, yields, for $1\le r\le i-1$,
\begin{equation*}
\sum_{y=1}^{i-1}\sin\!\Bigl(\frac{r\pi y}{i}\Bigr)
=\left\{
\begin{array}{cl}
0, & \text{if } r \text{ is even},\\
\cot\!\bigl(\frac{r\pi}{2i}\bigr)>0,
& \text{if } r \text{ is odd}.
\end{array}
\right.
\end{equation*}
Together with $\cos^{2m}(r\pi/i)\ge0$ and $\sin(r\pi/i)>0$, this shows that every term in the spectral sum above is
nonnegative. Therefore, retaining only the $r=1$ term in \eqref{3.2} and using the case $r=1$ of the identity above, we obtain
\begin{align}\label{lem-3.4-b}
&\mathbb{P}(\tau_i>2m)
\ge \frac{2}{i}\cos^{2m}\!\Bigl(\frac{\pi}{i}\Bigr)\sin\!\Bigl(\frac{\pi}{i}\Bigr)
\sum_{y=1}^{i-1}\sin\!\Bigl(\frac{\pi y}{i}\Bigr)\nonumber\\
&=\frac{2}{i}\cos^{2m}\!\Bigl(\frac{\pi}{i}\Bigr)\sin\!
\Bigl(\frac{\pi}{i}\Bigr)\cot\!\Bigl(\frac{\pi}{2i}\Bigr)
=\frac{4}{i}\cos^2\!\Bigl(\frac{\pi}{2i}\Bigr)\cos^{2m}\!\Bigl(\frac{\pi}{i}\Bigr).
\end{align}
Since $i\ge 2$, we have $\pi/(2i)\le\pi/4$ and hence $4\cos^2(\pi/(2i))\ge 2$.  Then by (\ref{lem-3.4-b}), we get (\ref{lem-3.4-a}). The proof is complete.
\end{proof}

\begin{proof}[Proof of Proposition \ref{pro-3.2}]
Fix $0<A<2/\pi^2$. For all sufficiently large $i$, set
$r_i:=\left\lfloor \frac{Ai^2\log\log i}{2}\right\rfloor$, $F_i:=\{T_i\ge 2r_i+2\}$, and
$E_i:=\{Y_i\ge 1,\ \tau_i^{(1)}>2r_i\}$. We take 
$i\ge 2$ large enough so that $r_i\ge 1$. 
By the definition of
$r_i$, we have
\begin{equation}\label{3.3}
2r_i+2>Ai^2\log\log i,\qquad
2r_i=Ai^2\log\log i+O(1).
\end{equation}
Since the increments $(T_i)_{i\ge 1}$ are independent and $F_i\in\sigma(T_i)$, the events  $(F_i)_{i\geq 1}$ are independent. Moreover, by \eqref{2.1}, on the event $E_i$ we have
\begin{equation*}
T_i\ge 1+(1+\tau_i^{(1)})>2r_i+2,
\end{equation*}
so $E_i\subset F_i$. By \eqref{2.2}, on the event $F_i$ we have $S_{i+1}>T_i$, and
 thus
\begin{equation}\label{3.4}
S_{i+1}>T_i\ge 2r_i+2>Ai^2\log\log i.
\end{equation}

Set $\alpha:=A\pi^2/2<1$ and choose $\varepsilon>0$ such that $\alpha+\varepsilon<1$. By \eqref{3.3} and
$\log\cos(\pi/i)=-\pi^2/(2i^2)+O(i^{-4})$, we have
\begin{equation}\label{3.4-1}
2r_i\log\cos\!\Bigl(\frac{\pi}{i}\Bigr)=(-\alpha+o(1))\log\log i.
\end{equation}
By the independence in Lemma \ref{lem-2.1}, Lemma \ref{lem-3.4}  and \eqref{3.4-1}, we have
\begin{align*}
&\mathbb{P}(F_i)
\ge \mathbb{P}(E_i)
= \mathbb{P}(Y_i\ge 1)\,\mathbb{P}(\tau_i^{(1)}>2r_i)\\
&\ge \frac{c}{1+c}\cdot \frac{2}{i}\cos^{2r_i}\!\Bigl(\frac{\pi}{i}\Bigr)
=\frac{2c}{(1+c)i}\exp\Bigl(2r_i\log\cos\!\Bigl(\frac{\pi}{i}\Bigr)\Bigr)\\
&=\frac{2c}{(1+c)i}\exp\bigl((-\alpha+o(1))\log\log i\bigr)
=\frac{2c}{(1+c)i}(\log i)^{-\alpha+o(1)}
\ge \frac{C_1}{i(\log i)^{\alpha+\varepsilon}},
\end{align*}
for some constant $C_1>0$. Since $\alpha+\varepsilon<1$, we have $\sum_i\mathbb{P}(F_i)=\infty$. By the
Borel-Cantelli lemma, $F_i$ occurs infinitely often almost surely. On this event, \eqref{3.4} implies that, for
infinitely many $i$,
\begin{equation*}
\frac{S_{i+1}}{(i+1)^2\log\log(i+1)}
> A\frac{i^2\log\log i}{(i+1)^2\log\log(i+1)},
\end{equation*}
 which together with $\lim_{i\to\infty}\frac{i^2\log\log i}{(i+1)^2\log\log(i+1)}=1$ implies that
\begin{equation*}
\limsup_{k\to\infty}\frac{S_k}{k^2\log\log k}\ge A
\quad\text{a.s.}
\end{equation*}
This is true for every $0<A<2/\pi^2$.  It follows that \eqref{pro-3.2-a} holds and the proof is complete.
\end{proof}

\section{Proof of Theorem \ref{thm-1.2}}\label{sec-4}

 In this section, we give the proof of Theorem \ref{thm-1.2}. We need the following two propositions, whose proofs are postponed to Sections \ref{sec-4.2} and \ref{sec-4.3} respectively.
Recall that $S_k$ is the inverse range of the once-reinforced random walk on $\mathbb Z$ with reinforcement parameter $c>0$.

\begin{pro}\label{pro-4.1}
It holds that
\begin{equation}\label{pro-4.1-a}
\liminf_{k\to\infty}\frac{S_k\log\log k}{k^2}\ge \frac{c^2}{2}\quad{\rm a.s.}
\end{equation}
\end{pro}

\begin{pro}\label{pro-4.2}
It holds that 
\begin{equation}\label{pro-4.2-a}
\liminf_{k\to\infty}\frac{S_k\log\log k}{k^2}\le \frac{c^2}{2}\quad{\rm a.s.}
\end{equation}
\end{pro}

\smallskip

\begin{proof}[Proof of Theorem \ref{thm-1.2}]
By Propositions \ref{pro-4.1} and \ref{pro-4.2}, we get
\begin{equation*}
\liminf_{k\to\infty}\frac{S_k\log\log k}{k^2}=\frac{c^2}{2}\quad{\rm a.s.,}
\end{equation*}
 which together with Lemma \ref{lem-2.4}(ii) implies that
\begin{equation*}
\limsup_{n\to\infty}\frac{R_n}{\sqrt{n\log\log n}}
= \left(\frac{c^2}{2}\right)^{-1/2}=\frac{\sqrt 2}{c}\quad {\rm a.s.}
\end{equation*}
Equivalently,
\begin{equation*}
\limsup_{n\to\infty}\frac{cR_n}{\sqrt{2n\log\log n}}=1\quad{\rm a.s.}
\end{equation*}
The proof is complete.
\end{proof}

\subsection{Proof of Proposition \ref{pro-4.1}}\label{sec-4.2}

We prove Proposition \ref{pro-4.1} by adapting the generating function argument in Serlet's hitting time estimates
\cite[Section~6, Theorem~17]{Serlet2013} to the hyperbolic generating function \eqref{2.3}.
We first record the following elementary result. 
For $h>0$ and integers $0\le a<k$, monotonicity and $0\le\tanh u\le1$ for $u\ge0$ give
\begin{align}\label{e:er}
0\le \int^{kh}_{ah}\tanh u\,  du-h\sum^{k-1}_{i=a}\tanh (ih)\le h.
\end{align}

\begin{lem}\label{lem-4.3}
If $\mu_k\to\infty$ and $\mu_k=o(k)$, then
\begin{equation*}
\log\Phi_k(\mu_k/k)=-2c\log\cosh(\mu_k/2)+o(\mu_k)=-c\mu_k+o(\mu_k)\quad \mbox{as}\ k\to\infty.
\end{equation*}
\end{lem}

\begin{proof}
It suffices to consider all sufficiently large $k$. For such $k$, the parameter $\theta_k:=\mu_k/k$ is positive and
converges to $0$. Applying \eqref{2.3} with $\theta=\theta_k$, we obtain
\begin{equation*}
\log\Phi_k(\mu_k/k)
=k\log\operatorname{sech}\theta_k-\sum_{i=1}^{k-1}\log\Bigl(1+c\tanh\theta_k\tanh(i\theta_k/2)\Bigr).
\end{equation*}
The contribution of the first term is $o(\mu_k)$  since
\begin{equation*}
k\log\operatorname{sech}\theta_k=O(k\theta_k^2)=O(\mu_k^2/k).
\end{equation*}
Moreover,
\begin{equation*}
\sup_{1\le i\le k-1}c\tanh\theta_k\tanh(i\theta_k/2)\le c\tanh\theta_k\to0.
\end{equation*}
Thus, by Taylor expansion of $\log(1+x)$ at $0$,  we have
\begin{align*}
&\sum_{i=1}^{k-1}\log\Bigl(1+c\tanh\theta_k\tanh(i\theta_k/2)\Bigr)
=c\tanh\theta_k\sum_{i=1}^{k-1}\tanh(i\theta_k/2)+O(k\tanh^2\theta_k)\\
&=c\theta_k\sum_{i=1}^{k-1}\tanh(i\theta_k/2)+O(\frac{\mu^2_k}{k}).
\end{align*}
Hence, by \eqref{e:er},
\begin{align*}
&\log\Phi_k(\mu_k/k)
=-c\theta_k\sum_{i=1}^{k-1}\tanh(i\theta_k/2)+O(\frac{\mu^2_k}{k})\\
&=-2c\frac{\mu_k}{2k}\sum_{i=1}^{k-1}\tanh(i\mu_k/2k)+O(\frac{\mu^2_k}{k})
=-2c\int_0^{\mu_k/2}\tanh u\,du+O(\frac{\mu^2_k}{k}+\frac{\mu_k}{k})\\
&=-2c\log\cosh(\mu_k/2)+o(\mu_k).
\end{align*}
Since $\log\cosh(\mu_k/2)=\mu_k/2+o(\mu_k)$,
the final equality in the statement follows.
\end{proof}

\begin{proof}[Proof of Proposition \ref{pro-4.1}]
Fix $q>1$ and let $k_n:=\lfloor q^n\rfloor$. Take $A<c^2/2$, and choose $b>0$ such that
$Ab^2/2-cb<-1$. Set $\mu_n:=b\log n$. Then $\mu_n\to\infty$ and $\mu_n=o(k_n)$. Recall that
$\Phi_{k_n}(\mu_n/k_n)=\mathbb{E}[(\operatorname{sech}(\mu_n/k_n))^{S_{k_n}}]$. Hence, for all large $n$, since
$\operatorname{sech}(\mu_n/k_n)<1$, by Markov's inequality and Lemma \ref{lem-4.3},  we have
\begin{align*}
&\mathbb{P}\Bigl(S_{k_n}\le A\frac{k_n^2}{\log n}\Bigr)
\le \bigl(\cosh(\mu_n/k_n)\bigr)^{A k_n^2/\log n}\Phi_{k_n}(\mu_n/k_n)\\
&=\exp\!\left(\frac{A\mu_n^2}{2\log n}-c\mu_n+o(\log n)\right)
=\exp\!\left(\Bigl(\frac{Ab^2}{2}-cb+o(1)\Bigr)\log n\right)
= n^{\frac{Ab^2}{2}-cb+o(1)},
\end{align*}
 where in the first two equalities,  we used $\log\cosh x=x^2/2+o(x^2)$ as $x\to0$, $\mu_n=b\log n$, and
$o(\log n)=o(1)\log n$. The right hand side is summable, and
then the Borel-Cantelli
lemma implies that
\begin{equation*}
\liminf_{n\to\infty}\frac{S_{k_n}\log n}{k_n^2}\ge A
\quad\text{a.s.}
\end{equation*}
Since $A<c^2/2$ is arbitrary and $\log\log k_n\sim\log n$,
taking $A<c^2/2$ to be rational and  letting $A\uparrow c^2/2$ along a
 sequence yield
\begin{equation*}
\liminf_{n\to\infty}\frac{S_{k_n}\log\log k_n}{k_n^2}\ge \frac{c^2}{2}
\quad\text{a.s.}
\end{equation*}

Now fix $k$ large and choose $n$ such that $k_n\le k<k_{n+1}$. Since $S_k$ is increasing, $\log\log k\ge \log\log k_n$,
and $k^2\le k_{n+1}^2$, we have
\begin{equation*}
\frac{S_k\log\log k}{k^2}
\ge \frac{S_{k_n}\log\log k_n}{k_{n+1}^2}
=\Bigl(\frac{k_n}{k_{n+1}}\Bigr)^2
\frac{S_{k_n}\log\log k_n}{k_n^2}.
\end{equation*}
Taking the liminf and using $k_n/k_{n+1}\to q^{-1}$, we obtain
\begin{equation*}
\liminf_{k\to\infty}\frac{S_k\log\log k}{k^2}\ge q^{-2}\frac{c^2}{2}
\quad\text{a.s.}
\end{equation*}
Finally, letting $q\downarrow 1$ along a
 sequence
yields  (\ref{pro-4.1-a}). The proof is complete.
\end{proof}

\subsection{Proof of Proposition \ref{pro-4.2}}\label{sec-4.3}

 In this subsection, we prove Proposition \ref{pro-4.2} by applying the exponential change of measure method in \cite[Section~2.2]{DZ2010} to the block generating function derived from \eqref{2.3}, and then using the Borel-Cantelli lemma for independent blocks. We first prove two lemmas.

\begin{lem}\label{lem-4.4}
Let $(k_n)_{n\ge 1}$ be a strictly increasing sequence of positive integers, and $(\mu_n)_{n\ge 2}$ be a sequence of positive real numbers satisfying
\begin{equation*}
k_n\to\infty,\quad \mu_n\to\infty,\quad \mu_n=o(k_n),\quad
\frac{k_{n-1}\mu_n}{k_n}\to 0 \quad \mbox{as}\ n\to\infty.
\end{equation*}
For $n\ge 2$, define
\begin{align}\label{4.2-a}
B_n:=S_{k_n}-S_{k_{n-1}}=\sum_{i=k_{n-1}}^{k_n-1}T_i.
\end{align}
Then the following block analogue of Lemma \ref{lem-4.3} holds:
\begin{align}\label{4.3}
\log\mathbb{E}\Bigl[\bigl(\operatorname{sech}(\mu_n/k_n)\bigr)^{B_n}\Bigr]
=-2c\log\cosh(\mu_n/2)+o(\mu_n)=-c\mu_n+o(\mu_n) \quad \mbox{as}\ n\to\infty.
\end{align}
\end{lem}

\begin{proof}
Let $\theta_n:=\mu_n/k_n$. By Lemma \ref{lem-2.1},  \eqref{2.8-a} and \eqref{4.2-a}, we have
\begin{align*}
\log\mathbb{E}\Bigl[\bigl(\operatorname{sech}\theta_n\bigr)^{B_n}\Bigr]
&=(k_n-k_{n-1})\log\operatorname{sech}\theta_n-\sum_{i=k_{n-1}}^{k_n-1}\log\Bigl(1+c\tanh\theta_n\tanh(i\theta_n/2)\Bigr).
\end{align*}
The first term is $O(\mu_n^2/k_n)=o(\mu_n)$  since $\log \operatorname{sech}\theta_n=-\theta_n^2/2+O(\theta_n^4)$ as $n\to\infty$.
Hence, by \eqref{e:er} and the Taylor expansion over the block $k_{n-1}\le i<k_n$ as in Lemma \ref{lem-4.3}, we get
\begin{align*}
&\sum_{i=k_{n-1}}^{k_n-1}
\log\Bigl(1+c\tanh\theta_n\tanh(i\theta_n/2)\Bigr)\\
&=c\theta_n\sum_{i=k_{n-1}}^{k_n-1}\tanh(i\theta_n/2)+O(\frac{\mu^2_n}{k_n})\\
&=2c\int_{k_{n-1}\mu_n/2k_n}^{\mu_n/2}\tanh u\,du+O(\frac{\mu^2_n}{k_n}+\frac{\mu_n}{k_n})\\
&=2c\log\cosh(\mu_n/2)+o(\mu_n),
\end{align*}
where in the last step we used $k_{n-1}\mu_n/k_n\to 0$.
 It follows that the first equality in (\ref{4.3}) holds. Since
$\log\cosh(\mu_n/2)=\mu_n/2+o(\mu_n)$
as $n\to\infty$, the second equality in (\ref{4.3}) holds too. The proof is complete.
\end{proof}

\smallskip

Suppose the assumptions of Lemma \ref{lem-4.4} are in force. We use the notation of Lemma \ref{lem-4.4}. For $n\ge 2$ and $\mu>0$, define
\begin{equation}\label{4.4}
H_n(\mu):=\log\mathbb{E}\Bigl[\bigl(\operatorname{sech}(\mu/k_n)\bigr)^{B_n}\Bigr].
\end{equation}
Following the standard exponential change of measure method
(see,~e.g., \cite[Section~2.2]{DZ2010}), define the tilted measure $Q_n$ by
\begin{align}\label{4.5}
\frac{dQ_n}{d\mathbb{P}}&:=\frac{(\operatorname{sech}(\mu_n/k_n))^{B_n}}{e^{H_n(\mu_n)}}.
\end{align}

\begin{lem}\label{lem-4.5}
Fix $\varepsilon\in(0,1/2)$ and set
\begin{equation}\label{4.6}
t_n:=\frac{ck_n^2}{\mu_n}.
\end{equation}
Then there exists $\gamma_\varepsilon>0$ such that, for all large $n$,
\begin{equation}\label{4.7}
Q_n(B_n\le (1-\varepsilon)t_n)\le e^{-\gamma_\varepsilon\mu_n},\qquad Q_n(B_n\ge (1+\varepsilon)t_n)\le e^{-\gamma_\varepsilon\mu_n}.
\end{equation}
In particular,
\begin{equation}\label{4.8}
Q_n\bigl((1-\varepsilon)t_n\le B_n\le (1+\varepsilon)t_n\bigr)\to 1.
\end{equation}
\end{lem}

\begin{proof}
We first treat the left tail. Let $\mu_n^+:=(1+\varepsilon)\mu_n$, and $Q_n^+$ be the tilted measure defined by
\begin{equation}\label{4.9}
\frac{dQ_n^+}{d\mathbb{P}}:=\frac{(\operatorname{sech}(\mu_n^+/k_n))^{B_n}}{e^{H_n(\mu_n^+)}}.
\end{equation}
By \eqref{4.5} and \eqref{4.9}, we have
\begin{align*}
\frac{dQ_n}{dQ_n^+}=\frac{dQ_n/d\mathbb{P}}{dQ_n^+/d\mathbb{P}}
&=e^{H_n(\mu_n^+)-H_n(\mu_n)}\left(\frac{\operatorname{sech}(\mu_n/k_n)}{\operatorname{sech}
(\mu_n^+/k_n)}\right)^{B_n}=e^{H_n(\mu_n^+)-H_n(\mu_n)}\left(\frac{\cosh(\mu_n^+/k_n)}{\cosh(\mu_n/k_n)}\right)^{B_n}.
\end{align*}
Since $\mu_n^+>\mu_n$ and $\cosh$ is increasing on $[0,\infty)$, for any $x>0$, we have
\begin{equation}\label{4.10}
\begin{aligned}
Q_n(B_n\le x)&=e^{H_n(\mu_n^+)-H_n(\mu_n)}\mathbb{E}_{Q_n^+}\!\left[\Bigl(\frac{\cosh(\mu_n^+/k_n)}{\cosh(\mu_n/k_n)}\Bigr)^{B_n}1_{\{B_n\le x\}}\right]\\
&\le\exp\!\left(H_n(\mu_n^+)-H_n(\mu_n)+x\bigl[\log\cosh(\mu_n^+/k_n)-\log\cosh(\mu_n/k_n)\bigr]\right).
\end{aligned}
\end{equation}
Since $\mu_n^+$ also satisfies the assumptions of Lemma \ref{lem-4.4}, by \eqref{4.3} and \eqref{4.4}, we have
\begin{equation}\label{4.11}
H_n(\mu_n^+)-H_n(\mu_n)=-c\mu_n^+ + c\mu_n+o(\mu_n)=-c\varepsilon\mu_n+o(\mu_n).
\end{equation}
Since $\mu_n=o(k_n)$ and $\log\cosh x=x^2/2+o(x^2)$ as $x\to0$, we have
\begin{align}\label{4.12}
\log\cosh(\mu_n^+/k_n)-\log\cosh(\mu_n/k_n)
=\frac{(\mu_n^+)^2-\mu_n^2}{2k_n^2}+o(\mu_n^2/k_n^2)
=\frac{(2\varepsilon+\varepsilon^2)\mu_n^2}{2k_n^2}+o(\mu_n^2/k_n^2).
\end{align}
Taking $x=(1-\varepsilon)t_n=(1-\varepsilon)ck_n^2/\mu_n$ in \eqref{4.10}, and using \eqref{4.6}, \eqref{4.11} and \eqref{4.12}, we have
\begin{equation}\label{4.13-a}
\begin{aligned}
&Q_n(B_n\le (1-\varepsilon)t_n)
\le \exp\!\left(-c\varepsilon\mu_n+(1-\varepsilon)c\mu_n\Bigl(\varepsilon+\frac{\varepsilon^2}{2}\Bigr)+o(\mu_n)\right)\\
&=\exp\!\Bigl(-\frac{c}{2}(\varepsilon^2+\varepsilon^3)\mu_n+o(\mu_n)\Bigr).
\end{aligned}
\end{equation}
Now we similarly consider the right tail.
Let $\mu_n^-:=(1-\varepsilon)\mu_n$, and $Q_n^-$ be the tilted measure defined by
\begin{equation}\label{4.13}
\frac{dQ_n^-}{d\mathbb{P}}:=\frac{(\operatorname{sech}(\mu_n^-/k_n))^{B_n}}{e^{H_n(\mu_n^-)}}.
\end{equation}
By \eqref{4.5} and \eqref{4.13}, we have
\begin{equation*}
\frac{dQ_n}{dQ_n^-}=\frac{dQ_n/d\mathbb{P}}{dQ_n^-/d\mathbb{P}}=e^{H_n(\mu_n^-)-H_n(\mu_n)}\left(\frac{\cosh(\mu_n^-/k_n)}{\cosh(\mu_n/k_n)}\right)^{B_n}.
\end{equation*}
Since $\mu_n^-<\mu_n$ and $\cosh$ is increasing on $[0,\infty)$, for any $x>0$, we have
\begin{equation}\label{4.14}
\begin{aligned}
&Q_n(B_n\ge x)=e^{H_n(\mu_n^-)-H_n(\mu_n)}\mathbb{E}_{Q_n^-}\!\left[\Bigl(\frac{\cosh(\mu_n^-/k_n)}{\cosh(\mu_n/k_n)}\Bigr)^{B_n}1_{\{B_n\ge x\}}\right]\\
&\le\exp\!\left(H_n(\mu_n^-)-H_n(\mu_n)+x\bigl[\log\cosh(\mu_n^-/k_n)-\log\cosh(\mu_n/k_n)\bigr]\right).
\end{aligned}
\end{equation}
Since $\mu_n^-$ also satisfies the assumptions of Lemma \ref{lem-4.4}, by \eqref{4.3} and \eqref{4.4}, we have
\begin{equation}\label{4.15}
H_n(\mu_n^-)-H_n(\mu_n)=-c\mu_n^- + c\mu_n+o(\mu_n)=c\varepsilon\mu_n+o(\mu_n).
\end{equation}
Since $\mu_n=o(k_n)$ and $\log\cosh x=x^2/2+o(x^2)$ as $x\to0$, we have
\begin{align} \label{4.16}
&\log\cosh(\mu_n^-/k_n)-\log\cosh(\mu_n/k_n)
=\frac{(\mu_n^-)^2-\mu_n^2}{2k_n^2}+o(\mu_n^2/k_n^2)\nonumber\\
&=\frac{(-2\varepsilon+\varepsilon^2)\mu_n^2}{2k_n^2}+o(\mu_n^2/k_n^2).
\end{align}
Taking $x=(1+\varepsilon)t_n=(1+\varepsilon)ck_n^2/\mu_n$ in \eqref{4.14}, and using \eqref{4.6}, \eqref{4.15} and \eqref{4.16}, we get
\begin{equation}\label{4.16-a}
\begin{aligned}
&Q_n(B_n\ge (1+\varepsilon)t_n)
\le \exp\!\left(c\varepsilon\mu_n-(1+\varepsilon)c\mu_n\Bigl(\varepsilon-\frac{\varepsilon^2}{2}\Bigr)+o(\mu_n)\right)\\
&=\exp\!\Bigl(-\frac{c}{2}(\varepsilon^2-\varepsilon^3)\mu_n+o(\mu_n)\Bigr).
\end{aligned}
\end{equation}

Since $\varepsilon\in(0,1/2)$, both $\varepsilon^2+\varepsilon^3$ and $\varepsilon^2-\varepsilon^3$ are positive. Choose
\begin{equation*}
\gamma_\varepsilon:= \frac{c}{4}\left((\varepsilon^2+\varepsilon^3)\wedge (\varepsilon^2-\varepsilon^3)\right).
\end{equation*}
Then  by \eqref{4.13-a}, \eqref{4.14} and \eqref{4.16-a}, we know that  \eqref{4.7} holds for all sufficiently large $n$.  Since  $\mu_n\to\infty$,
 by \eqref{4.7}, we have
$$
Q_n\bigl((1-\varepsilon)t_n<B_n< (1+\varepsilon)t_n\bigr)\to 1,
$$
which implies \eqref{4.8}.
The proof is complete.
\end{proof}

\smallskip

\begin{proof}[Proof of Proposition \ref{pro-4.2}]
Fix $A>c^2/2$. Let
\begin{equation*}
k_n:=n^n\quad(n\ge 1),\qquad \ell_n:=\log\log k_n\quad(n\ge 2).
\end{equation*}
Then $\ell_n=\log n+o(\log n)$. For $n\ge 2$, define
\begin{equation*}
\mu_n:=\frac{c\ell_n}{A},\qquad B_n:=S_{k_n}-S_{k_{n-1}}.
\end{equation*}
Then $\mu_n\to\infty$ and $\mu_n=o(k_n)$. Moreover,
\begin{equation*}
\frac{k_{n-1}}{k_n}\sim\frac{(n-1)^{n-1}}{n^n}=\frac{1}{n}\Bigl(1-\frac{1}{n}\Bigr)^{n-1}\sim \frac{e^{-1}}{n}.
\end{equation*}
Since $\mu_n\sim\frac{c}{A}\cdot\log n$, it follows that
\begin{equation*}
\frac{k_{n-1}\mu_n}{k_n}\sim \frac{ce^{-1}}{A}\frac{\log n}{n}\longrightarrow 0.
\end{equation*}
Thus Lemmas \ref{lem-4.4} and \ref{lem-4.5} apply.

Let $Q_n$ be the tilted measure
 defined by \eqref{4.5}.
Fix $\varepsilon\in(0,1/2)$ small enough such that
\begin{equation*}
\frac{(1+\varepsilon)c^2}{2A}<1.
\end{equation*}
For $n\geq 2$, define
\begin{equation}\label{4.17}
F_n:=\left\{(1-\varepsilon)\frac{Ak_n^2}{\ell_n}\le B_n\le (1+\varepsilon)\frac{Ak_n^2}{\ell_n}\right\}.
\end{equation}
By \eqref{4.6} and the choice $\mu_n=c\ell_n / A$, we have
\begin{equation*}
t_n=\frac{ck_n^2}{\mu_n}=\frac{Ak_n^2}{\ell_n}.
\end{equation*}
Thus \eqref{4.8} and \eqref{4.17} imply that
$Q_n(F_n)\to 1$ and  $\log Q_n(F_n)=o(1)=o(\ell_n)$.
Moreover, by \eqref{4.3} and \eqref{4.4},
$H_n(\mu_n)=-c\mu_n+o(\mu_n)=-\frac{c^2}{A}\ell_n+o(\ell_n)$.
Since $\mu_n=o(k_n)$,
$\log\cosh(\mu_n/k_n)=\mu_n^2/(2k_n^2)+o(\mu_n^2/k_n^2)$.
Then by \eqref{4.5} and \eqref{4.17}, setting $\alpha:=(1+\varepsilon)c^2/(2A)<1$, we obtain
\begin{align*}
&\log\mathbb{P}(F_n)
=
\log\left(e^{H_n(\mu_n)}\mathbb{E}_{Q_n}\!\left[(\cosh(\mu_n/k_n))^{B_n}1_{F_n}\right]\right)\\
& \ge \log\left(e^{H_n(\mu_n)}(\cosh(\mu_n/k_n))^{(1-\varepsilon)Ak_n^2/\ell_n}Q_n(F_n)\right)\\
&=
H_n(\mu_n) +(1-\varepsilon)\frac{Ak_n^2}{\ell_n}\log\cosh(\mu_n/k_n)+\log Q_n(F_n)\\
&\ge -\frac{c^2}{A}\ell_n+(1-\varepsilon)\frac{Ak_n^2}{\ell_n}\cdot\frac{\mu_n^2}{2k_n^2}+o(\ell_n)\\
&=-\frac{(1+\varepsilon)c^2}{2A}\ell_n+o(\ell_n)\\
&=-\alpha\ell_n+o(\ell_n)
=(-\alpha+o(1))\log n,
\end{align*}
 which implies that
\begin{equation*}
\mathbb{P}(F_n)\ge n^{-\alpha+o(1)}.
\end{equation*}
Choose $\delta>0$ such that $\alpha+\delta<1$,  and for all large $n$, $\mathbb{P}(F_n)\ge n^{-(\alpha+\delta)}$.  Thus $\sum_n\mathbb{P}(F_n)=\infty$. Moreover, $F_n$ is measurable with respect to $\sigma(T_i:k_{n-1}\le i<k_n)$. Since the
intervals $[k_{n-1},k_n)$ are disjoint and the increments $(T_i)_{i\ge 1}$ are independent, the events $(F_n)_{n\ge 2}$ are
independent. By the Borel-Cantelli lemma,
 we get
\begin{align}\label{4.19}
\mathbb{P}(F_n\ {\rm i.o.})=1.
\end{align}

On the other hand, \eqref{pro-3.1-a} implies that almost surely for all sufficiently large $n$,
\begin{equation*}
S_{k_{n-1}}\le \Bigl(\frac{2}{\pi^2}+\varepsilon\Bigr) k_{n-1}^2\log\log k_{n-1}.
\end{equation*}
Moreover,
\begin{equation*}
\frac{k_{n-1}^2\log\log k_{n-1}}{k_n^2/\ell_n}\le \Bigl(\frac{k_{n-1}}{k_n}\Bigr)^2\ell_n^2\longrightarrow 0.
\end{equation*}
Thus
\begin{equation}\label{4.20}
S_{k_{n-1}}=o(k_n^2/\ell_n) \quad {\rm a.s.}
\end{equation}

 By \eqref{4.17}-\eqref{4.20}, we get that for infinitely many $n$,
\begin{equation*}
S_{k_n}=S_{k_{n-1}}+B_n\le (1+2\varepsilon)\frac{Ak_n^2}{\ell_n} \quad {\rm a.s.,}
\end{equation*}
which implies
\begin{equation*}
\liminf_{n\to\infty}\frac{S_{k_n}\ell_n}{k_n^2}\le (1+2\varepsilon)A\quad {\rm a.s.}
\end{equation*}
Since $(k_n)_{n\ge1}$ is a subsequence, we obtain
\begin{equation*}
\liminf_{k\to\infty}\frac{S_k\log\log k}{k^2}\le (1+2\varepsilon)A\quad {\rm a.s.}
\end{equation*}
Finally, letting $\varepsilon\downarrow0$ and then $A\downarrow c^2/2$ along sequences,
we get  \eqref{pro-4.2-a}. The proof is complete.
\end{proof}

\bigskip

 \noindent {\bf\large Acknowledgments}\  This work was supported by
 National Natural Science Foundation of China (Grant Nos. 12171335, 12471139, 12301603) and  Simons Foundation (\#429343).

\end{document}